\documentclass[12pt,reqno]{amsart}

\usepackage{amssymb}
\usepackage{amsmath}
\usepackage{latexsym}
\usepackage{amsthm}

\newtheorem{theorem}{Theorem}[section]
\newtheorem{proposition}{Proposition}[section]
\newtheorem{corollary}{Corollary}[section]
\newtheorem{lemma}{Lemma}[section]
\theoremstyle{definition}

\theoremstyle{plain}
\theoremstyle{remark}
\newtheorem{remark}{\it Remark}[section]

\renewcommand{\div}{\mathop\mathrm{div}}
\newcommand{\vphi}{\varphi}

\begin{document}
\title
[Lower bounds for the spectrum of the Laplace and Stokes
operators]
 {Lower bounds for the spectrum of the Laplace and Stokes
operators}
\author[A. Ilyin]{Alexei A. Ilyin}
\address
{Keldysh Institute of Applied Mathematics, Moscow}
\email{ilyin@keldysh.ru}

\begin{abstract}
We prove Berezin--Li--Yau-type lower bounds with additional
term for the
eigenvalues of the  Stokes operator and improve the previously
known estimates for the Laplace operator.
Generalizations to higher-order operators are given.
\medskip\medskip\medskip

\textit{ Dedicated to Professor R.\,Temam on the occasion
of his 70th birthday}
\end{abstract}

\thanks
{
This work
was supported in part by the Russian Foundation for
Basic Research,
grant~nos.~09-01-00288 and 08-01-00099,
and by the
the RAS Programme no.1
`Modern problems of theoretical mathematics'
}

\subjclass{35P15, 35Q30.}

\keywords{Stokes operator,
Dirichlet Laplacian, lower bounds,  Navier--Stokes equations.}

\maketitle
\setcounter{equation}{0}
\section{Introduction}\label{S:Intro}

Sharp lower bounds for the sums of the first $m$
eigenvalues of the Dirichlet Laplacian
$$
-\Delta\vphi_k=\mu_k\vphi_k,
\qquad \vphi_k\vert_{\partial\Omega}=0
$$
were obtained in~\cite{Li-Yau}:
\begin{equation}\label{Li--Yau}
\sum_{k=1}^m\mu_k\ge\frac n{2+n}
\left(
\frac{(2\pi)^n}{\omega_n|\Omega|}
\right)^{2/n}m^{1+2/n}\,.
\end{equation}
Here $|\Omega|<\infty$ denotes the volume of
a domain $\Omega\subset\mathbb{R}^n$ and $\omega_n$ denotes
the volume of the unit ball in~$\mathbb{R}^n$.
It was shown in~\cite{L-W} that the
estimate~(\ref{Li--Yau}) is equivalent by means of the
Legendre transform to an earlier result of
Berezin~\cite{Berezin}.

In view of the classical H.\,Weyl asymptotic formula
$$
\mu_k\sim
\left(
\frac{(2\pi)^n}{\omega_n|\Omega|}
\right)^{2/n}k^{2/n}\quad\text{as $k\to\infty$},
$$
the coefficient of $m^{1+2/n}$ in~(\ref{Li--Yau})
is sharp,
However, an improvement of the Li--Yau bound with
additional term that is linear in $m$ was obtained
in~\cite{Melas}:
\begin{equation}\label{Melas-orig}
\sum_{k=1}^m\mu_k\ge\frac n{2+n}
\left(
\frac{(2\pi)^n}{\omega_n|\Omega|}
\right)^{2/n}m^{1+2/n}+
c_n\frac{|\Omega|}I\,m,
\end{equation}
where
\begin{equation}\label{I}
I=\int_\Omega x^2dx,
\end{equation}
and the constant $c_n$
depends only on the dimension: $c_n= c/(n+2)$
with $c$ being an absolute constant
(in fact, (\ref{Melas-orig}) holds with
$c=1/24$).
Of course $I$ can be replaced by
$I=\min_{a\in\mathbb{R}^n}\int_\Omega (x-a)^2dx$.

In the theory of the attractors for
the Navier--Stokes equations (see, for example,
\cite{B-V,CF88,T} and the references therein)
 lower bounds
for the sums of the eigenvalues $\{\lambda_k\}_{k=1}^\infty$
of the Stokes operator  are very important.
In the case of a smooth domain the
eigenvalue problem for the Stokes operator
reads:
 \begin{equation}\label{Stokessmooth}
\aligned
&-\Delta\, v_k \,+\, \nabla\, p_k\,=\,\lambda_kv_k,\\
&\div v_k\,=\,0,\,\,\,v_k\vert_{\partial\Omega }\,=\,0.
\endaligned
\end{equation}
Li--Yau-type lower bounds for the spectrum
of the Stokes operator were obtained in~\cite{I-FA2009}:
\begin{equation}\label{lowerStokes}
\sum_{k=1}^m\lambda_k\,\ge\,
\frac n{2+n}\left(
\frac{(2\pi)^n}{\omega_n(n-1)|\Omega|}
\right)^{2/n}m^{1+2/n}\,.
\end{equation}
The  coefficient of $m^{1+2/n}$ here is also sharp
in view of the asymptotic formula
(\cite{Babenko} ($n=3$), \cite{Metiv} ($n\ge2$)):
\begin{equation}\label{Stokesassymp}
\lambda_k\sim
\left(
\frac{(2\pi)^n}{\omega_n(n-1)|\Omega|}
\right)^{2/n}k^{2/n}\quad\text{as $k\to\infty$}.
\end{equation}

The main results of this paper are twofold.
First, we extend  the approach of~\cite{Melas}
to the case of the Stokes operator
and, secondly, we obtain the exact solution of the
corresponding minimization problem, thereby giving a
much better value of the constant $c_n$ in~(\ref{Melas-orig})
(in fact, the sharp value
 in the framework of the approach of~\cite{Melas}).

To describe the minimization problem we consider
in the case of the Laplacian
an $L_2$-orthonormal  family of functions
 $\{\vphi_k\}_{k=1}^m\in H^1_0(\Omega)$. Then
 the function $F(\xi)$
$$
F(\xi)=\sum_{k=1}^m|\widehat{\vphi}_k(\xi)|^2,
\qquad
\widehat{\vphi}_k(\xi)=
(2\pi)^{-n/2}\int_{\Omega}\vphi_k(x)e^{-i\xi x}dx
$$
satisfies $F(\xi)\le M=(2\pi)^{-n}|\Omega|$
(see~\cite{Li-Yau})
and the additional regularity property
which was found and  used in~\cite{Melas}:
$|\nabla F(\xi)|\le L=2(2\pi)^{-n}\sqrt{|\Omega I}$.
Here and in what follows $I$ is defined in~$(\ref{I})$.

For the Stokes operator we consider an
$\mathbf{L}_2$-orthonormal  family of vector functions
 $\{u_k\}_{k=1}^m\in \mathbf{H}^1_0(\Omega)$, with
 $\div u_k=0$.
 Then as we show in \S\ref{S:Orth-vec-fun}
  the corresponding function
 $F(\xi)=\sum_{k=1}^m|\widehat{u}_k(\xi)|^2$
 satisfies the conditions
 $F(\xi)\le M=(2\pi)^{-n}(n-1)|\Omega|$
 and
 $|\nabla F(\xi)|\le L=
 2(2\pi)^{-n}(n(n-1))^{1/2}\sqrt{|\Omega| I}$.

By orthonormality we always have
$\int_{\mathbb{R}^n}F(\xi)d\xi=m$,
and taking the first $m$ eigenfunctions of the
Laplace (or Stokes) operator
for the $\varphi_k$ (or the $u_k$, respectively) we get
$\int_{\mathbb{R}^n}|\xi|^2F(\xi)d\xi=
\sum_{k=1}^m\mu_k\ (=\sum_{k=1}^m\lambda_k)$,
and
$\sum_{k=1}^m\mu_k\ge\Sigma_M(m)$
(respectively, $\sum_{k=1}^m\lambda_k\ge\Sigma_M(m)$),
where $\Sigma_M(m)$ is the solution of the
minimization problem:
find $\Sigma_M(m)$
\begin{equation}\label{min-prob0}
\aligned
&\int_{\mathbb{R}^n}
|\xi|^2F(\xi)d\xi\to\inf:=\Sigma_M(m),
\quad\text{under the conditions}\\
&0\le F(\xi)\le M,\qquad \int_{\mathbb{R}^n} F(\xi)\,d\xi=m.
\endaligned
\end{equation}
It was shown in \cite{Li-Yau} that
the minimizer $F_*$ is radial and has the form
shown in Fig.~\ref{F*},
\begin{figure}[htb]\label{F*}
\unitlength=0.5mm
\linethickness{0.2mm}
\begin{center}
\begin{picture}(140,80)
\put(0,0){\vector(1,0){140}}
\put(0,0){\vector(0,1){70}}
\put(0,40){\line(1,0){60}}
\put(130,-8){$r$}
\put(60,-8){$r_*$}
\put(-10,40){$M$}
\multiput(60,0)(0,3){14}%
{\circle*{0.4}}
\end{picture}
\end{center}
\caption{ Minimizer $F_*(|\xi|)$}\label{F_*}
\end{figure}
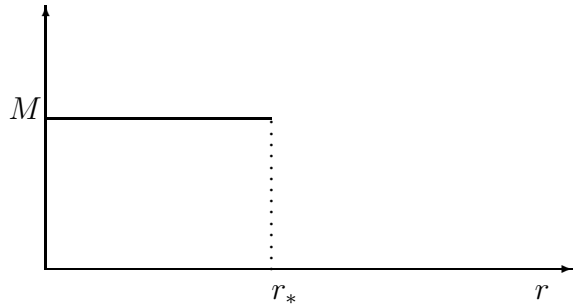
\medskip
where $r_*$  is defined by the condition
$
\int_{\mathbb{R}^n}F_*(|\xi|)d\xi=
m
$:
$$
\int_{\mathbb{R}^n}F_*(|\xi|)d\xi=
\sigma_{n}\int_0^{r_*}r^{n-1}F_*(r)dr=
\omega_{n}Mr_*^n=m.
$$
Then
$$
\Sigma_M(m)=
\int_{\mathbb{R}^n} |\xi|^2F_*(\xi)d\xi\,=\,
\sigma_{n}M\int_0^{r_*}r^{n+1}dr\,=\,
\frac{n}{n+2}
\biggl(\frac1{\omega_nM}\biggr)^{2/n}m^{1+2/n}
$$
giving~(\ref{Li--Yau}) upon substituting
$M=(2\pi)^{-n}|\Omega|$ for the Laplacian and
giving~(\ref{lowerStokes}) upon substituting
$M=(2\pi)^{-n}(n-1)|\Omega|$ for the Stokes operator
\cite{I-FA2009}.

The additional regularity property of $F(\xi)$:
$|\nabla F(\xi)|\le L$
 gives a better lower bound~\cite{Melas}:
$\sum_{k=1}^m\mu_k\ge\Sigma_{M,L}(m)$,
where $\Sigma_{M,L}(m)$ is the solution of the
minimization problem
\begin{equation}\label{min-prob1}
\aligned
\int_{\mathbb{R}^n} |\xi|^2F(\xi)d\xi\to\inf=:
\Sigma_{M,L}(m)
\quad\text{under the conditions},\\
0\le F(\xi)\le M,\qquad \int_{\mathbb{R}^n} F(\xi)\,d\xi=m,
 \qquad |\nabla F(\xi)|\le L.
\endaligned
\end{equation}
Clearly $\Sigma_{M,L}(m)\ge\Sigma_{M}(m)$ and
Lemma~1 in~\cite{Melas} (in the notation
our paper) reads:
\begin{equation}\label{LbdMelas}
\Sigma_{M,L}(m)\ge
\frac{n}{n+2}
\biggl(\frac1{\omega_nM}\biggr)^{2/n}m^{1+2/n}+
\frac{1}{6(n+2)}\,\frac{M^2}{L^2}\,m\,,
\end{equation}
giving~(\ref{Melas-orig}) with $c_n=1/(24(n+2))$
by substituting $M=(2\pi)^{-n}|\Omega|$
and $L=2(2\pi)^{-n}\sqrt{|\Omega|I}$.

In \S \ref{S:Min-problem} we find the exact
solution of the minimization problem~(\ref{min-prob1}):
$$
\Sigma_{M,L}(m)=
\frac{\sigma_nM^{n+3}}{(n+2)(n+3)L^{n+2}}
\bigl((t(m_*)+1)^{n+3}-t(m_*)^{n+3}\bigr),
$$
where $t(m_*)$ is the unique positive root
of the equation
$$
(t+1)^{n+1}-t^{n+1}=m_*,
\qquad m_*=
m\frac{(n+1)L^n}{\omega_nM^{n+1}}\,.
$$
We also find the  first three terms of
the asymptotic expansion of the solution
$\Sigma_{M,L}(m)$ in the following  descending  powers of $m$:
$m^{1+2/n}$, $m$, $m^{1-2/n}$, $m^{1-4/n}$, $\dots$.
Namely,
\begin{equation}\label{Sigma-asymp1}
\aligned
\Sigma_{M,L}(m)&=\Sigma_0(m)+O(m^{1-4/n}),\\
\Sigma_0(m)&=\frac{n}{n+2}
\biggl(\frac1{\omega_nM}\biggr)^{2/n}m^{1+2/n}+
\frac{n}{12}\,\frac{M^2}{L^2}\,m
-
\frac{n(n-1)(3n+2)}{1440}
\frac{M^4(M\omega_n)^{2/n}}{L^4}
m^{1-2/n},
\endaligned
\end{equation}
which shows that the second term is for all $n$
linear with respect to $m$ and positive
with coefficient that is  $n(n+2)/2$ times greater
than that in~(\ref{LbdMelas}),
while the third term is always negative.

Dropping the third term and using the expressions for
$M$ and $L$ we obtain the following
asymptotic lower bounds. Accordingly,
for large $m$ the
coefficient of $m$ in the second term on
the right-hand side in~(\ref{Lapasymp1})
is $n(n+2)/2$ times greater than that in~(\ref{Melas-orig}).

\begin{theorem}\label{T:L-S-asymp1}
The eigenvalues $\{\mu_k\}_{k=1}^\infty$ and
$\{\lambda_k\}_{k=1}^\infty$ of the Laplace
and Stokes operators satisfy the following
lower bounds:
\begin{eqnarray}
  \sum_{k=1}^m\mu_k&\ge& \frac n{n+2}
\biggl(\frac{(2\pi)^n}{\omega_n|\Omega|}\biggr)^{2/n}
m^{1+2/n}+\frac {n}{48}\frac {|\Omega|}{I}
\,m\,(1-\varepsilon_n(m)),\label{Lapasymp1}\\
  \sum_{k=1}^m\lambda_k &\ge&
\frac n{n+2}
\biggl(\frac{(2\pi)^n}{\omega_n(n-1)|\Omega|}\biggr)^{2/n}
m^{1+2/n}+\frac {(n-1)}{48}\frac {|\Omega|}{I}
\,m\,(1-\varepsilon_n(m)),\label{Stasymp1}
\end{eqnarray}
where $\varepsilon_n(m)\ge0$, $\varepsilon_n(m)=O(m^{-2/n})$.
\end{theorem}

Then in \S\ref{S:234} we turn to the analysis of the
particular cases $n=2,3,4$. The main result consists in
the explicit formulas for $\Sigma_{M,L}(m)$. The case $n=2$
is the simplest and we find
(see Lemma~\ref{l:Sigman=2}) the explicit formula for the
exact solution which coincides with the first three terms
of its asymptotic expansion
$$
\Sigma_{M,L}(m)=\Sigma_0(m)
=\frac1{2\pi M}\,m^2\,+\,\frac{M^2}{6L^2}\,m-
\frac{\pi M^5}{90 L^4}.
$$
For $n=3,4$ by means of the explicit formulas
in Lemmas~\ref{l:Sigman=3} and \ref{l:Sigman=4}
we show that
$$
\Sigma_{M,L}(m)>\Sigma_0(m).
$$
(The inequality
$
\Sigma_{M,L}(m)\ge\Sigma_0(m)
$
probably holds for any $n$, not only
for $n=2,3,4$.)
Then the negative contribution from the
third term in~(\ref{Sigma-asymp1}) is compensated
by a $(1-\beta)$-part of the positive second term
(where $0<\beta<1$ and $\beta$ is sufficiently close to $1$)
and we obtain the following theorem.
\begin{theorem}\label{T:L-S-asymp234}
The eigenvalues $\{\mu_k\}_{k=1}^\infty$ and
$\{\lambda_k\}_{k=1}^\infty$ for
$n=2,3,4$ satisfy:
\begin{eqnarray}
  \sum_{k=1}^m\mu_k&\ge& \frac n{n+2}
\biggl(\frac{(2\pi)^n}{\omega_n|\Omega|}\biggr)^{2/n}
m^{1+2/n}+\frac {n}{48}\beta^\mathrm{L}_n\frac {|\Omega|}{I}
\,m,\label{Lapasymp}\\
  \sum_{k=1}^m\lambda_k &\ge&
\frac n{n+2}
\biggl(\frac{(2\pi)^n}{\omega_n(n-1)|\Omega|}\biggr)^{2/n}
m^{1+2/n}+\frac {(n-1)}{48}
\beta^\mathrm{S}_n\frac {|\Omega|}{I}
\,m,\label{Stasymp}
\end{eqnarray}
where in the two-dimensional case
$
\beta^\mathrm{L}_2=\frac{119}{120}
$,
$
\beta^\mathrm{S}_2=\frac{239}{240},
$
while for $n=3,4$
it suffices to take
$
\beta^\mathrm{L}_3=0.986,
$
$
\beta^\mathrm{S}_3=0.986
$
and
$
\beta^\mathrm{L}_4=0.983,
$
$
\beta^\mathrm{S}_4=0.978.
$
\end{theorem}

Finally, in \S\ref{S:bi-L} we prove two-term
lower bounds for the Dirichlet bi-Laplacian.

\begin{remark}\label{Rem:Weidl}
{\rm
Two term lower bounds for the
2D Laplacian with the second term
of growth higher than linear  in
$m$ were obtained in~\cite{Weidl}.
They depend on the shape of $\partial\Omega$.
}
\end{remark}

\setcounter{equation}{0}
\section{Estimates for orthonormal
 vector functions}\label{S:Orth-vec-fun}

Throughout $\Omega$ is an open subset of $\mathbb{R}^n$
with finite $n$-dimensional Lebesgue measure $|\Omega|$:
$$
\Omega\subset\mathbb{R}^n,\ n\ge2,\qquad|\Omega|<\infty.
$$
We recall the functional definition of
the Stokes operator \cite{CF88,Lad,TNS}:
 $\mathcal{V}$ denotes the set of smooth
divergence-free vector functions with compact supports
$$
\mathcal{V}=\{u:\Omega\to\mathbb{R}^n,
\ u\in \mathbf{C}^\infty_0(\Omega),\
\div u=0\}
$$
and $H$ and $V$ are the the closures of
$\mathcal{V}$ in  $\mathbf{L}_2(\Omega)$ and
$\mathbf{H}^1(\Omega)$, respectively. The
Helmholtz--Leray orthogonal projection $P$ maps
$\mathbf{L}_2(\Omega)$ onto $H$,
$P:\mathbf{L}_2(\Omega)\to H$.
We have (see~\cite{TNS})
\begin{eqnarray}\label{HandV}
\mathbf{L}_2(\Omega)&=&H\oplus H^\perp,
\quad H^\perp=\{u\in\mathbf{L}_2(\Omega),
u=\nabla p,\ p\in L_2^{\mathrm{loc}}(\Omega)\},\\
\nonumber
V&\subseteq&\{u\in\mathbf{H}^1_0(\Omega),\ \div u=0\},
\end{eqnarray}
where  the last inclusion becomes equality
for a bounded $\Omega$ with Lipschitz boundary.
The Stokes operator $A$ is defined by the relation
\begin{equation}\label{Stokes}
(Au,v)=(\nabla u,\nabla v) \quad
\text{for all $u,v$ in $V$}
\end{equation}
and is an isomorphism between $V$ and $V'$.
For a sufficiently smooth $u$
$$
Au=-P\Delta u.
$$
The Stokes operator $A$ is
an unbounded self-adjoint positive operator in $H$
with discrete spectrum
$\{\lambda_k\}_{k=1}^\infty$, $\lambda_k\to\infty$ as
$k\to\infty$:
\begin{equation}\label{spectrum}
Av_k=\lambda_kv_k, \quad
0<\lambda_1\le\lambda_2\le\dots\,,
\end{equation}
where $\{v_k\}_{k=1}^\infty\in V$
are the corresponding orthonormal eigenvectors.
Taking the scalar product with $v_k$ we have by
orthonormality and~(\ref{Stokes}) that
\begin{equation}\label{lambdak}
    \lambda_k=\|\nabla v_k\|^2.
\end{equation}
In case when $\Omega$
is a bounded domain with smooth boundary the eigenvalue
problem~(\ref{spectrum})
goes over  to~(\ref{Stokessmooth}).

We recall that a family
 $\{\vphi_i\}_{i=1}^m\in L_2(\Omega)$ is called
suborthonormal~\cite{G-M-T}
if for any  $\zeta\in\mathbb{C}^m$
\begin{equation}\label{suborth}
\sum_{i,j=1}^m\zeta_i\zeta_j^*(\vphi_i,\vphi_j)\le
\sum_{j=1}^m|\zeta_j|^2.
\end{equation}

\begin{lemma}\label{l:bessel}
Any suborthonormal family $\{\vphi_i\}_{i=1}^m$
satisfies Bessel's inequality:
\begin{equation}\label{bessel}
\sum_{k=1}^mc_k^2\le\|f\|_{L_2(\Omega)}^2,
\quad\text{where}\quad c_k=(\vphi_k,f).
\end{equation}
\end{lemma}
\begin{proof}
Given an suborthonormal system~$\{\vphi_i\}_{i=1}^m$ (with
supports in $\Omega$), we build it up to a orthonormal
system $\{\psi_i\}_{i=1}^m\in L_2(\mathbb{R}^n)$ of the
form $\psi_k=\vphi_k+\chi_k$,
$\chi_k(x)=\sum_{j=1}^ma_{kj}\omega_j(x)$, where
$\{\omega_i\}_{i=1}^m$ is an arbitrary orthonormal system
with supports in $\mathbb{R}^n\setminus\Omega$. The
condition $(\psi_k,\psi_l)=\delta_{kl}$ is satisfied if the
we chose for the matrix $a=a_{ij}$ the symmetric
non-negative matrix $a=b^{1/2}$, where
$b_{ij}=\delta_{ij}-(\vphi_i,\vphi_j)$ (in view
of~(\ref{suborth}), $b$ is non-negative).

The system $\{\psi_i\}_{i=1}^m$ classically satisfies
Bessel's inequality, and since
$(\psi_k,f)=(\vphi_k,f)$, this gives~(\ref{bessel}).
\end{proof}

Suborthonormal families typically arise as a result of the
action of an orthogonal projection~\cite{G-M-T}.
\begin{lemma}\label{l:vec_suborth1}
If
$\{\vphi_k\}_{k=1}^m$ is orthonormal and
$P$ is an orthogonal projection, then
both families
$\eta_k=P\vphi_k$ and
$\xi_k=(I-P)\vphi_k$
 are suborthonormal.
\end{lemma}

We now obtain some estimates for the Fourier transforms
for  (sub)orthonormal families.
\begin{lemma}\label{l:Fourier-bessel}
If
$\{\vphi_k\}_{k=1}^m$ is suborthonormal, then
\begin{equation}\label{Fourier-bessel}
\sum_{k=1}^m|\widehat{\vphi}_k(\xi)|^2\le
(2\pi)^{-n}|\Omega|.
\end{equation}
\end{lemma}
\begin{proof}
This follows from~(\ref{bessel}) with
$f(x)=f_\xi(x)=(2\pi)^{-n/2}e^{-i\xi x}$.
\end{proof}
\begin{corollary} If  the family of vector functions
$\{u_k\}_{k=1}^m$ is orthonormal
in $\mathbf{L}_2(\Omega)$, then
\begin{equation}\label{n}
\sum_{k=1}^m|\widehat {u}_k(\xi)|^2\le(2\pi)^{-n}\, n|\Omega|.
\end{equation}
\end{corollary}
\begin{proof}
By Lemma~\ref{l:vec_suborth1} for  each
$j=1,\dots,n$ the
 family
$\{u_k^j\}_{k=1}^m$ is suborthonornal
and~(\ref{n}) follows from Lemma~\ref{l:Fourier-bessel}.
\end{proof}

The next lemma~\cite{I-FA2009} is essential for the Li--Yau
bounds for the Stokes operator and says  that under the
additional condition~$\div u_k=0$ the factor $n$ in
the previous estimate is replaced by $n-1$.
\begin{lemma}\label{l:n-1}
If the family of vector functions $\{u_k\}_{k=1}^m$ is
orthonormal and $u_k\in H$, then
\begin{equation}\label{n-1}
\sum_{k=1}^m|\widehat {u}_k(\xi)|^2\le
(2\pi)^{-n}\, (n-1)|\Omega|.
\end{equation}
\end{lemma}
\begin{proof}
First we observe that
$
\xi\cdot \widehat {u}_k(\xi)=
(2\pi)^{-n/2}\,i\int u_k\cdot\nabla_xe^{-i\xi x}\,dx=0
$
for all $\xi\in \mathbb{R}^n_\xi$
since the $u_k$'s are orthogonal to gradients
(see~(\ref{HandV})).
Let $\xi_0\ne0$ be of the form:
\begin{equation}\label{xi0}
\xi_0=(a,0,\dots,0), \qquad a\ne0.
\end{equation}
Since $\xi_0\cdot \widehat {u}_k(\xi_0)=0$, it follows that
$\widehat u_k^1(\xi_0)=0$ for  $k=1,\dots,m$. Hence,
by~(\ref{Fourier-bessel})
$$
\sum_{k=1}^m|\widehat {u}_k(\xi_0)|^2=
\sum_{j=2}^n\sum_{k=1}^m|\widehat {u}_k^j(\xi_0)|^2
\le (2\pi)^{-n}\,(n-1)|\Omega|.
$$
The general case reduces to the case~(\ref{xi0}) by the
corresponding rotation  of
$\mathbb{R}^n$
about the origin represented by the
orthogonal $(n\times n)-$matrix $\rho$.
Given a vector function $u(x)=(u^1(x),\dots,u^n(x))$ we
consider the vector function
$$
u_\rho(x)=\rho\, u(\rho^{-1}x),
\qquad x\in\rho\Omega.
$$
A straight forward calculation gives that
$
\div u_\rho(x)=\div u(y)$, where
$\rho^{-1}x=y$. In addition,
$ (u_\rho,v_\rho)=(u,v)$.
Combining this we obtain that the family
$\{(u_k)_\rho\}_{k=1}^m$ is orthonormal and belongs to
$H(\rho\Omega)$.

Next we calculate $\widehat{u_\rho}$
and see that
$
\widehat{u_\rho}(\xi)=
\rho\widehat{u}(\rho^{-1}\xi).
$
We now fix an arbitrary
$\xi\in\mathbb{R}^n$, $\xi\ne0$ and
set
$\xi_0=(|\xi|,0,\dots,0)$. Let $\rho$ be the
 rotation such that $\xi=\rho^{-1}\xi_0$.
Then we have
$$
\sum_{k=1}^m|\widehat {u}_k(\xi)|^2=
\sum_{k=1}^m|\widehat {u}_k(\rho^{-1}\xi_0)|^2=
\sum_{k=1}^m|\rho^{-1}\widehat {(u_k)_\rho}(\xi_0)|^2=
\sum_{k=1}^m|\widehat {(u_k)_\rho}(\xi_0)|^2
\le(2\pi)^{-n}\, (n-1)|\Omega|,
$$
where we have used
 that inequality~(\ref{n-1}) has been proved for $\xi$ of
the form~(\ref{xi0}) for any orthonormal family
of divergence-free vector functions.
Finally, the estimate (\ref{n-1}) is extended to $\xi=0$ by
continuity.
\end{proof}

For the orthonormal family
$\{u_k\}_{k=1}^m\in H$ we set
\begin{equation}\label{Fstokes}
F_\mathrm{S}(\xi)=\sum_{k=1}^m|\widehat{u}_k(\xi)|^2.
\end{equation}
\begin{lemma}\label{l:gradF}
The following inequality holds:
\begin{equation}\label{gradF}
|\nabla F_\mathrm{S}(\xi)|\le2(2\pi)^{-n}(n(n-1))^{1/2}
\sqrt{|\Omega|I}.
\end{equation}
\end{lemma}
\begin{proof}
The proof is similar to that in~\cite{Melas}
for the Laplacian. We have
$$
\partial_{j}\widehat{u}_k^l(\xi)=
-(2\pi)^{-n/2}i\int_{\Omega} u_k^l(x)x_je^{-i\xi x}dx.
$$
Since the family $\{u_k^l\}_{k=1}^m$ is subothonormal,
by Lemma~\ref{l:Fourier-bessel} we have
$$
\sum_{k=1}^m|\partial_{j}\widehat{u}_k^l(\xi)|^2\le
(2\pi)^{-n}\int_{\Omega} x_j^2dx
$$
and
$$
\sum_{k=1}^m|\nabla\widehat{u}_k(\xi)|^2\le
(2\pi)^{-n}\,n\int_{\Omega}x^2dx=
(2\pi)^{-n}\,nI(\Omega),
\qquad I(\Omega)=\int_\Omega x^2dx.
$$
Next, using~(\ref{n-1}) we obtain
$$
|\nabla F_\mathrm{S}(\xi)|\le2
\biggl(\sum_{k=1}^m|\widehat{u}_k(\xi)|^2\biggr)^{1/2}
\biggl(\sum_{k=1}^m|\nabla\widehat{u}_k(\xi)|^2\biggr)^{1/2}\le
2(2\pi)^{-n}(n(n-1))^{1/2}
\sqrt{|\Omega|I}.
$$
\end{proof}
If, in addition, the orthonormal family
$\{u_k\}_{k=1}^m$ belongs to
$V\subseteq\{u\in\mathbf{H}^1_0(\Omega),\ \div u=0\}$,
then, by the Plancherel theorem, the function
$F_\mathrm{S}(\xi)$ defined in~(\ref{Fstokes}) satisfies
\begin{equation}\label{enumSt}
\aligned
0\le F_\mathrm{S}(\xi)\le
    M_\mathrm{S}&=(2\pi)^{-n}(n-1)|\Omega|;\\
|\nabla F_\mathrm{S}(\xi)|\le
    L_\mathrm{S}&=2(2\pi)^{-n}(n(n-1))^{1/2}
\sqrt{|\Omega|I};\\
\int F_\mathrm{S}(\xi)\,d\xi&=m;\\
\int |\xi|^2F_\mathrm{S}(\xi)\,d\xi&=
            \sum_{k=1}^m\|\nabla u_k\|^2.
\endaligned
\end{equation}

In the case of the Laplace operator, that is,
for an orthonormal family
$\{\vphi_k\}_{k=1}^m\in H^1_0(\Omega)$ the corresponding
function
$
F_\mathrm{L}(\xi)=\sum_{k=1}^m|\widehat{\vphi}_k(\xi)|^2
$
satisfies \cite{Li-Yau}, \cite{Melas}
\begin{equation}\label{enumLap}
\aligned
0\le F_\mathrm{L}(\xi)\le
    M_\mathrm{L}&=(2\pi)^{-n}|\Omega|;\\
|\nabla F_\mathrm{L}(\xi)|\le
    L_\mathrm{L}&=2(2\pi)^{-n}\sqrt{|\Omega|I};\\
\int F_\mathrm{L}(\xi)\,d\xi&=m;\\
\int |\xi|^2F_\mathrm{L}(\xi)\,d\xi&=
            \sum_{k=1}^m\|\nabla \vphi_k\|^2.
\endaligned
\end{equation}

\setcounter{equation}{0}
\section{Minimization problem}\label{S:Min-problem}

There is not much difference now between the Laplace and
the Stokes operators, and the problem
of lower bounds for the eigenvalues reduces to the
problem of finding $\Sigma_{M,L}(m)$ defined in
the minimization problem~(\ref{min-prob1}).

We consider the symmetric-decreasing rearrangement
$F^*(\xi)$ of the $F(\xi)$. It is well known
(see, for example, \cite{Talenti}) that
$0\le F^*(\xi)\le M$,
$\int F^*(\xi)\,d\xi=\int F(\xi)\,d\xi=m$
and, in addition,
$|\nabla F^*(\xi)|\le \mathrm{ess\ sup}|\nabla F(\xi)|$.
Also,
\begin{equation}\label{H-L}
\int |\xi|^2F(\xi)d\xi\ge\int |\xi|^2F^*(\xi)d\xi.
\end{equation}
This inequality follows from the Hardy--Littlewood inequality
$$
\int G(\xi)F(\xi)d\xi\le
\int G^*(\xi)F^*(\xi)d\xi,
$$
where $G(\xi)=G^*(\xi)=R^2-|\xi|^2$ and
without loss of generality we assume that
the ball $B_R$ contains the supports of
$F$ and $F^*$.

Thus, we obtain a one-dimensional problem
equivalent
to~(\ref{min-prob1}):
\begin{equation}\label{onemin}
\aligned
&\sigma_n\int_0^\infty r^{n+1}F(r)dr\to\inf=:
\Sigma_{M,L}(m),\\
0\le F(r)\le M,
\qquad &\sigma_n\int_0^\infty r^{n-1}F(r)dr=m,\qquad
-F'(r)\le L,
\endaligned
\end{equation}
where $F(r)$ is decreasing and without loss of generality
we assume that $F$ is absolutely continuous.

We consider the function $\Phi_s(r)$
shown in Fig.~\ref{Fs}:
\begin{equation}\label{Phi}
\Phi_s(r)=\left\{
\begin{array}{lll}
    M, &\text{for}\ {0\le r\le s;} \\
    M-Lr, &\text{for}\ {s\le r\le s+ M/L;} \\
    0, &\text{for}\ {s+M/L\le r.} \\
\end{array}
\right.
\end{equation}
\begin{lemma}\label{l:min}
Suppose that $\int_0^\infty r^\alpha\Phi_s(r)dr=m^*$ and
$\beta\ge\alpha$. Then for any decreasing
absolutely continuous function $F$ satisfying
the conditions
$$
0\le F\le M,\qquad
\int_0^\infty r^\alpha F(r)dr=m^*,
\qquad
-F'\le L,
$$
the following inequality holds:
\begin{equation}\label{alphabeta}
\int_0^\infty r^\beta F(r)dr\ge
\int_0^\infty r^\beta\Phi_s(r)dr.
\end{equation}
\end{lemma}
\begin{proof}
If $F$ is an admissible function and
$F(s)=\Phi_s(s)\ (=M)$, then
$F\equiv\Phi_s$. Hence for any admissible function
$F$ such that
$F\ne \Phi_s$ (and, hence,
$F(r_0)<M=\Phi_s(r_0)$  at some point $r_0$,
$0\le r_0 < s$), the graph of $F$
intersects the graph of $\Phi_s$
to the right of $r_0$
at exactly one point with $r$-coordinate $a$,
where $a$ is in the region $s<a<s+M/L$.
In other words,
 $F(r)\le \Phi_s(r)$ for $0\le r\le a$ and
$F(r)\ge \Phi_s(r)$ for $a\le r<\infty$.
 Therefore
$$
\aligned
\int_0^ar^\beta(\Phi_s(r)-F(r))dr\le
a^{\beta-\alpha}\int_0^ar^\alpha(\Phi_s(r)-F(r))dr=\\
a^{\beta-\alpha}\int_a^\infty r^\alpha(F(r)-\Phi_s(r))dr\le
\int_a^\infty r^\beta(F(r)-\Phi_s(r))dr,
\endaligned
$$
where the  functions under the integral sings are
non-negative.
\end{proof}
\begin{lemma}\label{l:calcul}
By a straight forward calculation
\begin{equation}\label{calcul}
\int_0^\infty r^\gamma\Phi_s(r)dr=
\frac{M^{\gamma+2}}{(\gamma+1)(\gamma+2)L^{\gamma+1}}
\bigl((t+1)^{\gamma+2}-t^{\gamma+2}\bigr),
\qquad s=\frac{tM}L.
\end{equation}
\end{lemma}

Combining the above results we see that the minimizing
function is given by~(\ref{Phi})
\begin{figure}[htb]
\unitlength=0.5mm
\linethickness{0.2mm}
\begin{center}
\begin{picture}(140,80)
\put(60,40){\line(1,-1){39.8}}
\put(0,0){\vector(1,0){140}}
\put(0,0){\vector(0,1){70}}
\put(0,40){\line(1,0){60}}
\put(135,-8){$r$}
\put(60,-8){$s$}
\put(-10,40){$M$}
\put(90,-8){$s+M/L$}
\put(70,38){$\text{slope } -\!L$}
\multiput(60,0)(0,3){14}%
{\circle*{0.4}}
\end{picture}
\end{center}
\caption{ Minimizer $\Phi_{s}(|\xi|)$}
\label{Fs}
\end{figure}
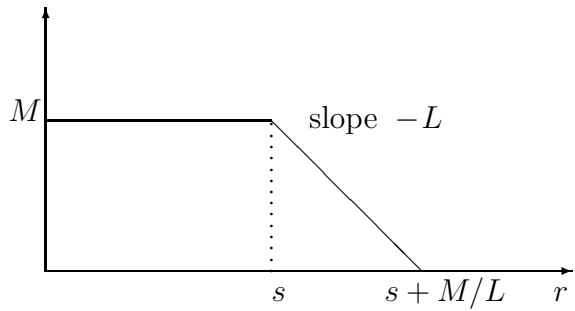
and the second condition
in~(\ref{onemin}) becomes
$\sigma_n\int_0^\infty r^{n-1}\Phi_s(r)dr=m$, which
in view of~(\ref{calcul}) gives the equation for $t$
(and $s$):
\begin{equation}\label{eq-for-t}
(t+1)^{n+1}-t^{n+1}=
m\frac{n(n+1)L^n}{\sigma_nM^{n+1}}=
m\frac{(n+1)L^n}{\omega_nM^{n+1}}=:m_*.
\end{equation}
It will be shown (see~(\ref{m*})) that for $m\ge1$ the
right-hand side in~(\ref{eq-for-t}) is greater than $1$.
Since the left-hand side is a polynomial of order $n$
(with positive coefficients) monotonely increasing from $1$
to $\infty$ on $\mathbb{R}^+$, the equation~(\ref{eq-for-t})
has a unique solution $t=t(m_*)\ge 0$.
Using~(\ref{calcul}) this time  with $\gamma=n+1$
we  find the solution of~(\ref{min-prob1}),
that is, $\Sigma_{M,L}(m)$. In other words, we have just
proved the following result.
\begin{proposition}\label{l:min-prob}
The solution of the minimization
problem~$(\ref{min-prob1})$ is
given by
\begin{equation}\label{Sigma}
\Sigma_{M,L}(m)=
\frac{\sigma_nM^{n+3}}{(n+2)(n+3)L^{n+2}}
\bigl((t(m_*)+1)^{n+3}-t(m_*)^{n+3}\bigr),
\end{equation}
where $t(m_*)$ is the unique positive root of the
equation~$(\ref{eq-for-t})$.
\end{proposition}

\begin{remark}\label{Rem:Weidl1}
{\rm The shape of the minimizer~(\ref{Phi}) was found
in~\cite{Weidl}. We use it here to find the exact
solution~(\ref{Sigma}) of the minimization
problem~(\ref{min-prob1}).
}
\end{remark}

We give explicit expressions for $\Sigma_{M,L}(m)$
(and thereby explicit lower bounds for sums
of eigenvalues of the Laplace and Stokes operators)
for the dimension $n=2,3,4$ in \S\ref{S:234}.
Meanwhile we obtain the asymptotic expansion
for~$\Sigma_{M,L}(m)$ valid for all dimensions $n$.

First, it is convenient to write
the right-hand side in~(\ref{eq-for-t}) in the form
\begin{equation}\label{eta}
(t+1)^{n+1}-t^{n+1}=
(\eta+1/2)^{n+1}-(\eta-1/2)^{n+1},
\qquad \eta=t+1/2,
\end{equation}
since this substitution kills half of the coefficients
in the explicit expression for the polynomial.
Then the equation~(\ref{eta}) takes the form
$$
(n+1)\biggl(\eta^n+\frac{n(n-1)}{24}\eta^{n-2}+
\frac{n(n-1)(n-2)(n-3)}{1920}\eta^{n-4}+\dots
\biggr)=m_*.
$$
The unique positive root $\eta(m_*)$ of this equation
has the asymptotic expansion
\begin{equation}\label{eta-asymp}
\eta(m_*)=\biggl(\frac{m_*}{n+1}\biggr)^{1/n}-
\frac{n-1}{24}\biggl(\frac{m_*}{n+1}\biggr)^{-1/n}+
\frac{(n-1)(n-3)(2n-1)}{5760}
\biggl(\frac{m_*}{n+1}\biggr)^{-3/n}+\dots\,.
\end{equation}
The first term here is obvious, the second and the third
terms can be found in the standard way.
Therefore substituting~(\ref{eta-asymp}) into
the second factor in~(\ref{Sigma}) we obtain
\begin{equation}\label{sigma-asymp}
\aligned
(t(m_*)+1)^{n+3}-t(m_*)^{n+3}=
(\eta(m_*)+1/2)^{n+3}-(\eta(m_*)-1/2)^{n+3}=\\
(n+3)\biggl[\biggl(\frac{m_*}{n+1}\biggr)^{1+2/n}+
\frac{(n+2)}{12}\frac{m_*}{n+1}-
\frac{(n-1)(n+2)(3n+2)}{1440}
\biggl(\frac{m_*}{n+1}\biggr)^{1-2/n}+\dots\biggr],
\endaligned
\end{equation}
and then~(\ref{Sigma}) along with the expression for
$m_*$ in (\ref{eq-for-t}) finally gives~(\ref{Sigma-asymp1}).

\begin
{proof}[Proof of Theorem~\ref{T:L-S-asymp1}]
The difference between the Laplace and Stokes
operators is now only in the definition of
$M$ and $L$ and we consider the case of the Stokes
operator.
Since
$$
  \sum_{k=1}^m\|\nabla u_k\|^2=
\int |\xi|^2F_\mathrm{S}(\xi)\,d\xi\ge
\Sigma_{M,L}(m),
$$
it remains to substitute into~(\ref{Sigma-asymp1})
$ M_\mathrm{S}$ and  $L_\mathrm{S}$ from~(\ref{enumSt}).
This gives that $ \sum_{k=1}^m\|\nabla u_k\|^2\ge
\text{r.\,h.\,s\, of~(\ref{Stasymp1})}$
and inequality~(\ref{Stasymp1}) follows by taking
the first normalized eigenvectors of the Stokes problem
for the $u_k$'s. The proof of~(\ref{Lapasymp1})
is totally similar.
\end{proof}

We conclude this section by checking that both for the
Laplace and Stokes operators $m_*\ge1$, that is,
\begin{equation}\label{m*}
    \frac{(n+1)L^n}{\omega_nM^{n+1}}\,\ge\,1.
\end{equation}
(Geometrically this means that $\Phi_s$ always
has a horizontal part.)
This follows from the inequality
\begin{equation}\label{IgeOm}
I=\int_\Omega |x|^2dx\ge
\frac {n|\Omega|^{1+2/n}}{(n+2)\omega_n^{2/n}}\,,
\end{equation}
which, in turn, is~(\ref{H-L}) with $F$ being the
characteristic function of $\Omega$.
In fact,~(\ref{IgeOm}) and the formulas for $M$ and $L$
 give much more than~(\ref{m*}):
\begin{equation}\label{m*n}
m_*\ge m_0^{\mathrm{L}}=\frac{(n+1)(4\pi)^n}{\omega_n^2}
\biggl(\frac n{n+2}\biggr)^{n/2},\
m_*\ge m_0^{\mathrm{S}}=\frac{(n+1)(4\pi)^n}{(n-1)\omega_n^2}
\biggl(\frac {n^2}{(n-1)(n+2)}\biggr)^{n/2}
\end{equation}
for the Laplace and Stokes operators, respectively,
in the sense that the right-hand sides in~(\ref{m*n})
tend to infinity as $n\to\infty$.

\setcounter{equation}{0}
\section{Lower bounds for the Laplace and Stokes operators
for $n=2,3,4$}\label{S:234}

\subsection*{The case $n=2$}
The two-dimensional case is the simplest and the results
are the most complete.
\begin{lemma}\label{l:Sigman=2}
In the two-dimensional case
\begin{equation}\label{Sigma2}
\Sigma_{M,L}(m)=
\frac1{2\pi M}\,m^2\,+\,\frac{M^2}{6L^2}\,m-
\frac{\pi M^5}{90 L^4}.
\end{equation}
\end{lemma}
\begin{proof} In view of~(\ref{Sigma}) we only need to
calculate the last factor there. The positive root
$t(m_*)$ of the equation~$(\ref{eq-for-t})_{n=2}$,
which is the quadratic equation $(t+1)^3-t^3=m_*$, is
\begin{equation}\label{root2}
t(m_*)=\sqrt{\frac{m_*}{3}-\frac1{12}}-\frac12
\end{equation}
and using~(\ref{eta}) we obtain
$$
(t(m_*)+1)^{5}-t(m_*)^{5}=
\frac59\,m_*^2+\frac59\,m_*-\frac19\,.
$$
The rest is a direct substitution. We note that
$\Sigma_{M,L}(m)=\Sigma_0(m)_{n=2}$, see~(\ref{Sigma-asymp1}).
\end{proof}
\begin{theorem}\label{T:n=2}
For $n=2$
the eigenvalues of the Laplace and Stokes operators satisfy
\begin{eqnarray}
  \sum_{k=1}^m\mu_k&\ge& \frac{2\pi}{|\Omega|}\,m^2+
 \frac1{24}\frac{|\Omega|}I\,m \biggl(1-\frac1{120m}\biggr)
  \ge
  \frac{2\pi}{|\Omega|}\,m^2+ \frac1{24}\,
  \frac{119}{120}\frac{|\Omega|}I\,m,
  \label{Lapnonasymp2}\\
  \sum_{k=1}^m\lambda_k &\ge&
\frac{2\pi}{|\Omega|}\,m^2+
 \frac1{48}\frac{|\Omega|}I\,m \biggl(1-\frac1{240m}\biggr)
  \ge
  \frac{2\pi}{|\Omega|}\,m^2+\frac1{48}\,
  \frac{239}{240}\frac{|\Omega|}I\,m.\label{Stnonasymp2}
\end{eqnarray}
\end{theorem}
\begin{proof}
We consider~(\ref{Lapnonasymp2}).
In view of~(\ref{enumLap}) we have
$M=M_\mathrm{L}=(2\pi)^{-2}|\Omega|$ and
$L=L_\mathrm{L}=2(2\pi)^{-2}\sqrt{|\Omega|I}$,
therefore (\ref{Sigma2}) gives for the Laplacian
$$
\aligned
\sum_{k=1}^m\mu_k\ge\Sigma_{M,L}(m)=
\frac{2\pi}{|\Omega|}\,m^2+
 \frac1{24}\frac{|\Omega|}I\,m-
\frac1{90\cdot 2^6\pi}\frac{|\Omega|^3}{I^2}\ge
\frac{2\pi}{|\Omega|}\,m^2+
 \frac1{24}\frac{|\Omega|}I\,m-
\frac1{90\cdot 2^5}\frac{|\Omega|}{I}\,,
\endaligned
$$
where the last inequality follows from~(\ref{IgeOm}):
$|\Omega|^2/I\le2\pi$.
The proof~(\ref{Stnonasymp2}) is similar:
$M=M_\mathrm{S}=(2\pi)^{-2}|\Omega|$,
$L=L_\mathrm{S}=2(2\pi)^{-2}\sqrt{2}\sqrt{|\Omega|I}$ and
by~(\ref{Sigma2})
$$
\aligned
\sum_{k=1}^m\lambda_k\ge\Sigma_{M,L}(m)=
\frac{2\pi}{|\Omega|}\,m^2+
 \frac1{48}\frac{|\Omega|}I\,m-
\frac1{90\cdot 2^8\pi}\frac{|\Omega|^3}{I^2}\ge
\frac{2\pi}{|\Omega|}\,m^2+
 \frac1{48}\frac{|\Omega|}I\,m-
\frac1{90\cdot 2^7}\frac{|\Omega|}{I}\,.
\endaligned
$$
The proof of this theorem (which is
Theorem~$\ref{T:L-S-asymp234}_{\, n=2}$) is complete.
\end{proof}
\subsection*{The case $n=4$}
\begin{lemma}\label{l:Sigman=4}
In the four-dimensional case
\begin{equation}\label{Sigma4}
\Sigma_{M,L}(m)\ge\frac{8\sqrt{2}}{3\pi M^{1/2}}\,m^{3/2}+
\frac13\cdot\beta
\frac{M^2}{L^2}\,m,
\end{equation}
where $\beta=\beta^\mathrm{L}_4=0.983$ for the Laplace operator
 and
$\beta=\beta^\mathrm{S}_4=0.978$ for the Stokes operator.
\end{lemma}
\begin{proof}
The positive root
$t(m_*)$ of the equation~$(\ref{eq-for-t})_{n=4}$
(which is biquadratic with respect to $\eta=t+1/2$) is
$$
t(m_*)=\sqrt{\sqrt{20m_*+5}/10-1/4}-1/2
$$
and with the help of~(\ref{eta}) we find that
$$
\aligned
\sigma(m_*):=(t(m_*)+1)^{7}-t(m_*)^{7}=
(7/50)\bigl(m_*\sqrt{20m_*+5}+
5m_*-\sqrt{20m_*+5}+15/7\bigr)>\\
\frac7{50}\biggl(2\sqrt{5}m_*^{3/2}+
5m_*-\frac{7\sqrt{5}}4m_*^{1/2}+\frac{15}7
-\frac{17\sqrt{5}}{64}m_*^{-1/2}\biggr)>
\frac{7\sqrt{5}}{25}m_*^{3/2}+
\frac7{10}m_*-\frac{49\sqrt{5}}{200}m_*^{1/2},
\endaligned
$$
where we used the inequality
$1+x/2-x^2/8<\sqrt{1+x}<1+x/2$ and the fact that
$m_*\ge1$.
We observe that the three terms on the right here
are as in $(\ref{sigma-asymp})_{n=4}$ so that
$\Sigma_{M,L}(m)>\Sigma_0(m)_{n=4}$,
see~(\ref{Sigma-asymp1}).

We now take advantage of the fact that
$m_*$  is large, namely,
$m_*\ge m_0^{\mathrm{L}}=(5/9)2^{12}=2275.5\dots$ and
$m_*\ge m_0^{\mathrm{S}}=5\cdot2^{16}/3^5=1348.7\dots$,
respectively, (see (\ref{m*n})). The smallest constant
$\alpha>0$ such that
$$
\alpha m_*\ge\frac{49\sqrt{5}}{200}m_*^{1/2},\quad
m_*\in[m_0,\infty)
$$
clearly is
$\alpha_0=(49\sqrt{5}/{200})m_0^{-1/2}$.
For the Laplace operator
$\alpha_0^\mathrm{L}=
(49\sqrt{5}/{200})(m_0^\mathrm{L})^{-1/2}=0.01148\dots$,
while for the Stokes operator
$\alpha_0^\mathrm{S}=0.01491\dots$.
Hence
$$
\sigma(m_*)>
\frac{7\sqrt{5}}{25}m_*^{3/2}+
\frac7{10}\beta m_*,\qquad
\beta=1-\frac{10}7\alpha,
$$
where
$\beta^\mathrm{L}=0.9835\dots$ and
$\beta^\mathrm{S}=0.9786\dots$, respectively,
and~(\ref{Sigma4})
follows  by going over from $m_*$ to
$m$ (see~(\ref{eq-for-t}), (\ref{Sigma}),
(\ref{Sigma-asymp1})).
\end{proof}

\begin
{proof}[Proof of
 Theorem~$\ref{T:L-S-asymp234}_{\, n=4}$]
We substitute the expressions for $M$ and $L$
into~(\ref{Sigma4}) and get the result.
\end{proof}

\subsection*{The case $n=3$}
\begin{lemma}\label{l:Sigman=3}
In the tree-dimensional case
\begin{equation}\label{Sigma3}
\Sigma_{M,L}(m)\ge
\frac35\biggl(
\frac3{4\pi M}^{2/3}\biggr)m^{5/3}+
\frac14\cdot\beta
\frac{M^2}{L^2}m\,,
\end{equation}
where $\beta=\beta^\mathrm{L}_3=0.9869$ and
$\beta=\beta^\mathrm{S}_3=0.9861$
for the Laplace and Stokes operators,
respectively.
\end{lemma}
\begin{proof}
The unique positive root
$t(m_*)$ of the cubic equation~$(\ref{eq-for-t})_{n=3}$
is given by Cardano's formula
(in which all the roots are taken positive)
$$
t(m_*)=\frac12\biggl(m_*+\sqrt{m_*^2+\frac1{27}}\biggr)^{1/3}
-\frac12\biggl(-m_*+\sqrt{m_*^2+\frac1{27}}\biggr)^{1/3}-
\frac12\,.
$$
By a direct substitution using~(\ref{eta})
we have
$$
\aligned
\sigma(&m_*):=(t(m_*)+1)^{6}-t(m_*)^{6}=\\
&\frac1{48}
\bigl(3\sqrt{3+81m_*^2}+27m_* \bigr)^{2/3}
\bigl(11m_*-\sqrt{3+81m_*^2} \bigr)+\\
&\frac{5}8m_*+\\
&\frac1{48}\biggl(
\bigl(3\sqrt{3+81m_*^2}-27m_* \bigr)^{2/3}
\bigl(11m_*+\sqrt{3+81m_*^2} \bigr)-
7\bigl(3\sqrt{3+81m_*^2}+27m_* \bigr)^{1/3}
\biggr)+\\
&\frac7{48}\bigl(3\sqrt{3+81m_*^2}-27m_* \bigr)^{1/3},
\endaligned
$$
where the four  terms above are written in the order
$m_*^{5/3}$, $m_*$, $m_*^{1/3}$, $m_*^{-1/3}$. We now
obtain a  lower bound for $\sigma(m_*)$. Using
the inequality $\sqrt{1+x}<1+x/2$ below
we get that the first term is greater than
$$
\frac{3\cdot2^{2/3}}8\,m_*^{5/3}-\frac{2^{2/3}}{32}\,m_*^{-1/3}.
$$
The third term is equal to
$$
-\frac{90m_*+12\sqrt{3+81m_*^2}}
{48(3\sqrt{3+81m_*^2}+27m_*)^{2/3}}>
-\frac{198m_*+2/m_*}{48(54m_*)^{2/3}}=
-\frac{11\cdot2^{1/3}}{48}m_*^{1/3}-
\frac{2^{1/3}}{48\cdot 9}\,m_*^{-5/3}.
$$
The fourth term is equal to
$$
\frac7{16(3\sqrt{3+81m_*^2}+27m_*)^{1/3}}>
\frac{7m_*^{-1/3}}{48\cdot 2^{1/3}}
\biggl(1+\frac1{27m_*^2}\biggr)^{-1/3}>
\frac{7\cdot 2^{1/3}}{96}
m_*^{-1/3},
$$
since $m_*\ge1$.
Collecting these estimates we obtain
\begin{equation}\label{lwrbdsigma}
\sigma(m_*)>\frac{3\cdot2^{2/3}}8\,m_*^{5/3}+
\frac{5}8m_*-
\frac{11\cdot2^{1/3}}{48}m_*^{1/3},
\end{equation}
so that as for $n=4$ we have
$\Sigma_{M,L}(m)>\Sigma_0(m)_{n=3}$,
see~(\ref{Sigma-asymp1}).

As in Lemma~\ref{l:Sigman=4} we
have from~(\ref{m*n}) that
$m_*\ge m_0^{\mathrm{L}}=
(16\cdot 27\pi/5)(3/5)^{1/2}=210.2\dots$ and
$m_*\ge m_0^{\mathrm{S}}=72\pi(9/10)^{3/2}=193.1\dots$
for the Laplace and Stokes
operators, respectively.
 Therefore
the inequality
$$
\alpha m_*-
\frac{11\cdot2^{1/3}}{48}m_*^{1/3}
\ge0,\quad m_*\in[m_0,\infty)
$$
is satisfied for all
$
\alpha\ge\alpha_0=
\frac{11\cdot2^{1/3}}{48}m_0^{-2/3}.
$
Hence for  the Laplace operator
$\alpha_0^\mathrm{L}=0.008165\dots$,
while for the Stokes operator
$\alpha_0^\mathrm{S}=0.008641\dots$.
Hence
$$
\sigma(m_*)>\frac{3\cdot2^{2/3}}8\,m_*^{5/3}+
\frac{5}8\beta m_*,\qquad
\beta=1-\frac{8}5\alpha,
$$
where
$\beta^\mathrm{L}=0.9869\dots$ and
$\beta^\mathrm{S}=0.9861\dots$, respectively,
which proves~(\ref{Sigma3}) (see~(\ref{Sigma-asymp1})).
\end{proof}

\begin
{proof}[Proof of
 Theorem~$\ref{T:L-S-asymp234}_{\, n=3}$]
The proof immediately follows from~(\ref{Sigma3}).
The proof of Theorem~\ref{T:L-S-asymp234} is
complete.
\end{proof}

\setcounter{equation}{0}
\section{Further examples. Dirichlet bi-Laplacian}
\label{S:bi-L}

Other elliptic equations and systems
with constant coefficients and Dirichlet boundary
conditions can be treated quite similarly.
We restrict ourselves to the Dirichlet
bi-Laplacian:
\begin{equation}\label{biLap}
\Delta^2\varphi_k=\nu_k\varphi_k,
\qquad\vphi_k\vert_{\partial\Omega}=0,
\quad\frac{\vphi_k}{\partial n}\vert_{\partial\Omega}=0.
\end{equation}
We consider
the $L_2$-orthonormal  family of eigenfunctions
 $\{\vphi_k\}_{k=1}^m\in H^2_0(\Omega)$. Then
 the function
$
F(\xi)=\sum_{k=1}^m|\widehat{\vphi}_k(\xi)|^2
$
satisfies the same three
conditions:
\begin{equation}\label{condbiLap}
1)\ 0\le F(\xi)\le M,\qquad 2)\ |\nabla F(\xi)|\le L,\qquad
3)\ \int_{\mathbb{R}^2}F(\xi)\,d\xi=m,
\end{equation}
where as before
$M=(2\pi)^{-n}|\Omega|$ and
$ L=2(2\pi)^{-n}\sqrt{|\Omega I}$.
Since $\sum_{k=1}^m\nu_k=
\int_{\mathbb{R}^n}|\xi|^4F(\xi)\,d\xi$,
we have to find  the solution $\Sigma_{M,L}^4(m)$ of the
minimization problem
\begin{equation}\label{minbiLap}
\int_{\mathbb{R}^2} |\xi|^4f(\xi)d\xi\to\inf=:
\Sigma_{M,L}^4(m)\quad
\text{under conditions~(\ref{condbiLap})},
\end{equation}
whose solution is found similarly to
Proposition~\ref{l:min-prob}.
\begin{proposition}\label{p:min-prob-biLap}
The solution of the minimization
problem~$(\ref{minbiLap})$ is
given by
\begin{equation}\label{Sigma-n4}
\Sigma_{M,L}^4(m)=
\frac{\sigma_nM^{n+5}}{(n+4)(n+5)L^{n+4}}
\bigl((t(m_*)+1)^{n+5}-t(m_*)^{n+5}\bigr),
\end{equation}
where $t(m_*)$ is the unique positive root of the
equation~$(\ref{eq-for-t})$.
\end{proposition}
\begin{proof}
The minimizer~(\ref{Phi}) and the
equation for $s$ (\ref{eq-for-t}) are the same as before.
It remains to calculate the integral
$\int_{\mathbb{R}^n}|\xi|^4\Phi_s(|\xi|)\,d\xi$
based on Lemma~\ref{l:calcul}.
\end{proof}

We restrict ourselves to the least
technical  two-dimensional case.
\begin{lemma}\label{l:Sigma24}
For $n=2$ the exact solution $\Sigma_{M,L}^4(m)$
can be found explicitly:
\begin{equation*}\label{Sigma24}
\Sigma_{M,L}^4(m)=
\frac1{3\pi^2 M^2}\,m^3\,+\,\frac{M}{3\pi L^2}\,m^2-
\frac{\pi M^7}{7\cdot 3^4 L^6}\,.
\end{equation*}
\end{lemma}
\begin{proof}
As before the unique  positive root
$t(m_*)$ of the equation $(t+1)^3-t^3=m_*$ is
given by~(\ref{root2}):
$t(m_*)=\sqrt{{m_*}/{3}-1/2}-1/2$,
and a direct substitution gives
\begin{equation}\label{m*7}
(t(m_*)+1)^{7}-t(m_*)^{7}=
\frac7{27}\,m_*^3+\frac79\,m_*^2-\frac1{27}\,.
\end{equation}
It remains to substitute
(\ref{m*7}) into~(\ref{Sigma-n4}) with
 $$
m_*=m\frac{(n+1)L^n}{\omega_nM^{n+1}}{\ \vert_{n=2}}\,=\,
m\frac{3L^2}{\pi M^3}.
$$
\end{proof}

\begin{theorem}\label{T:biLap}
For $n=2$ the eigenvalues of the Dirichlet bi-Laplacian
satisfy
\begin{equation}\label{biLapn=2}
\sum_{k=1}^m\nu_k\ge \frac{16\pi^2}{3|\Omega|^2}\,m^3+
\frac{\pi}{3I}
\,m^2 \biggl(1-\frac1{7\cdot 3^3\cdot 2^6\,m^2}\biggr)
\ge
\frac{16\pi^2}{3|\Omega|^2}\,m^3+
\frac{\pi}{3\,I}\,
\frac{12095}{12096}\,m^2.
\end{equation}
\end{theorem}
\begin{proof}
Similar to Theorem~\ref{T:n=2}.
\end{proof}

\begin{remark}\label{Rem:sharp}
{\rm The coefficient of the leading term $m^3$ in
(\ref{biLapn=2}) is sharp.
}
\end{remark}

\end{document}